\documentclass[11pt, twoside]{article}
\usepackage{amssymb, amsmath, amsthm}
\usepackage[left=20mm, right=20mm, bottom=25mm, top=25mm]{geometry}
\usepackage{inputenc}
\usepackage{fancyhdr}
\usepackage{tikz}
\usetikzlibrary{positioning}
\usepackage{titling}
\usepackage{hyperref}
\predate{}
\postdate{}
\usepackage{lipsum}
\usepackage{titling}
\usepackage{xcolor}
\usepackage{placeins}
\usepackage{tikz}
\usepackage{verbatim}
\usepackage{titling}
\usepackage[T1]{fontenc}
\usepackage{currvita}
\usepackage{bbm}
\usepackage{enumitem}
\usepackage{cleveref}
\newtheorem{theorem}{Theorem}
\newtheorem{lemma}[theorem]{Lemma}
\newtheorem{proposition}[theorem]{Proposition}

\newtheorem{conjecture}[theorem]{Conjecture}
\newtheorem{corollary}[theorem]{Corollary}

\newtheorem{claim888}{Claim}

\newtheorem{case1}{Case}
\newtheorem{case2}{Case}

\theoremstyle{definition}

\theoremstyle{remark}

\newcommand{\cB}{\mathcal{B}}
\newcommand{\cA}{\mathcal{A}}

\newcommand{\cP}{\mathcal{P}}

\newcommand{\B}{\mathcal{B}}

\newcommand{\G}{\mathcal{G}}

\begin{document}
\newcommand{\Addresses}{
\bigskip
\footnotesize

\medskip

\noindent Maria-Romina~Ivan, \textsc{Department of Pure Mathematics and Mathematical Statistics, Centre for Mathematical Sciences, Wilberforce Road, Cambridge, CB3 0WB, UK.}\par\noindent\nopagebreak\textit{Email addresses: }\texttt{mri25@dpmms.cam.ac.uk}

\medskip

\noindent Sean~Jaffe, \textsc{Department of Pure Mathematics and Mathematical Statistics, Centre for Mathematical Sciences, Wilberforce Road, Cambridge, CB3 0WB, UK.}\par\noindent\nopagebreak\textit{Email address: }\texttt{scj47@cam.ac.uk}}

\pagestyle{fancy}
\fancyhf{}
\fancyhead [LE, RO] {\thepage}
\fancyhead [CE] {MARIA-ROMINA IVAN AND SEAN JAFFE}
\fancyhead [CO] {THE EXACT SATURATION NUMBER FOR THE DIAMOND}
\renewcommand{\headrulewidth}{0pt}
\renewcommand{\l}{\rule{6em}{1pt}\ }
\title{\Large{\textbf{THE EXACT SATURATION NUMBER FOR THE DIAMOND}}}
\author{MARIA-ROMINA IVAN AND SEAN JAFFE}
\date{ }
\maketitle
\begin{abstract}
What is the smallest size of a family of subsets of $[n]$ such that it does not contain an induced copy of $Q_2$ as a poset (known as the \textit{diamond}), but adding a new set creates such a copy? It is easy to see that a maximal chain has this property, and thus the answer is at most $n+1$. Despite the simplicity of the diamond structure, the lower bound stagnated at $\sqrt n$ for quite some time, until recently the authors obtained a linear lower bound. In this paper, we fully solve this question showing that such a family must have size at least $n+1$.
\end{abstract}

\section{Introduction}
We say that a poset $(\mathcal Q, \leq')$ contains an \textit{induced copy} of a poset $(\mathcal P, \leq)$ if there exists an injective function $f:\mathcal P\rightarrow\mathcal Q$ such that $(\mathcal P, \leq)$ and $(f(\mathcal P), \leq')$ are isomorphic. The poset $(\mathcal Q, \leq')$ is thought as the host environment. In what follows, all posets are finite. The canonical environment for finite posets is the power of $[n]=\{1,2,\dots,n\}$ equipped with the partial order given by set inclusion.

Let $\mathcal P$ be a fixed poset. We say that a family $\mathcal F$ of subsets of $[n]=\{1,2,\dots,n\}$ is $\mathcal P$-\textit{saturated} if $\mathcal F$ does not contain an induced copy of $\mathcal P$, but for every subset $S$ of $[n]$ such that $S\notin\mathcal F$, the family $\mathcal F\cup \{S\}$ contains such a copy. We denote by $\text{sat}^*(n, \mathcal P)$ the size of the smallest $\mathcal P$-saturated family of subsets of $[n]$. In general, we refer to $\text{sat}^*(n, \mathcal P)$ as the \textit{induced saturation number} of $\mathcal P$.

This notion, inspired by saturation for graphs, was introduced by Ferrara, Kay, Kramer, Martin, Reiniger, Smith and Sullivan \cite{ferrara2017saturation}. Despite their similarity, graph saturation and poset saturation are vastly different, partially due to the rigidity of \textit{induced} copies of posets. The dominant conjecture in the area is the following:
\begin{conjecture}[\cite{keszegh2021induced}] Let $\mathcal P$ be a finite poset. Then, either $\text{sat}^*(n,\mathcal P)=O(1)$, or $\text{sat}^*(n,\mathcal P)=\Theta(n)$.    
\end{conjecture}
The saturation number for posets has already been shown to display a dichotomy. Keszegh, Lemons, Martin, P\'alv\"olgy and Patk\'os \cite{keszegh2021induced} showed that the saturation number is bounded or at least $\log_2(n)$. This was later improved by Freschi, Piga, Sharifzadeh and Treglown \cite{freschi2023induced} who showed that $\text{sat}^*(n,\mathcal P)$ is bounded or at least $2\sqrt n$. In the other direction, Bastide, Groenland, Ivan and Johnston \cite{polynomial} showed that the saturation number of any poset grows at most polynomially. For the general case, this is the state of the art: polynomial upper bound and, if not bounded, $\sqrt n$ lower bound.

It is worth mentioning that, from a structural viewpoint, not only do we not know what features make a poset have unbounded saturation number, but currently we do not even have a good guess as to what those might be. A step in this direction has been recently taken by Ivan and Jaffe who showed that for any poset, one can add at most 3 `special' points to obtain a poset with saturation number at most linear \cite{gluing}. They also showed that the linear sum of two posets has unbounded saturation number if at least one of the posets possesses certain structural properties (which guarantee it has unbounded saturation number) -- this is the first, and so far the only, `new from old' operation which appears to be amenable to analysing saturation numbers. 

What about specific posets? It turns out that even then, determining the rate of growth of the saturation number is far from easy, and the analysis is highly dependent on the precise structure of the poset in question. In fact, linearity has only been established for a few posets, most notably the butterfly \cite{ivan2020saturationbutterflyposet, keszegh2021induced}, the antichain \cite{bastide2024exact}, the complete multipartite poset that is not a chain \cite{gluing, sums}, the $\mathcal N$ poset \cite{N}, and the diamond \cite{diamondlinear}. Out of all of them, the diamond, denoted by $\mathcal D_2$, (2-dimensional Boolean algebra, Hasse diagram depicted below) was the most resistant to analysis.
\begin{figure}[hbt!]
\hspace{0.1cm}\\
\centering
\begin{tikzpicture}
\node (top) at (4,1) {$\bullet$}; 
\node (left) at (3,0) {$\bullet$};
\node (right) at (5,0) {$\bullet$};
\node (bottom) at (4,-1) {$\bullet$};
\draw (top) -- (left) -- (bottom) -- (right) -- (top);
\node at (4,-2) 
{Hasse diagram of the diamond poset ($\mathcal D_2$).};
\end{tikzpicture}

\end{figure}
\FloatBarrier
Whilst it is easy to see that a maximal chain is diamond-saturated, and hence $\text{sat}^*(n,\mathcal D_2)\leq n+1$, the lower bound has stagnated at $\sqrt n$ for years. Recently, the authors showed that $\text{sat}^*(n,\mathcal D_2)\geq\frac{n+1}{5}$, which established the conjectured linearity. But what is the \textit{exact} value of $\text{sat}^*(n,\mathcal D_2)$? In this paper we fully settle this question.
\begin{theorem}
    \label{maintheorem}
Let $n\in\mathbb N$. Then $\text{sat}^*(n,\mathcal D_2)=n+1$.
\end{theorem}
The proof is very different from the one that just established linearity, although there are some similarities, especially in the setup. At the heart of the proof is the idea of separating, in a diamond-saturated family, the minimal sets, the maximal sets, and the collection of sets that are the middle points of a diamond created when adding a set from the outside to the family. The last category of sets will be further split into two, one to be analysed with the minimal elements and the other with the maximal elements. The second crucial idea is to construct a sequence of nested families, where the starting one is the family of minimal (maximal) sets, that systematically isolates all the elements of the ground set that appear in at least one minimal (avoid at least one maximal) set.

The plan of the paper is as follows. In Section 2 we introduce all notations and establish some short but helpful associated lemmas. In Section 3 we upper bound the set of minimal elements together with a suitably chosen subset of the `middle' elements, as described above. In Section 4, combining the bounds and the structure analysed in previous sections, we prove our main result.

Throughout the paper our notation is standard. For a finite non-empty set $X$ and a positive integer $k\leq |X|$ we denote by $\binom{X}{k}$ the collection of subsets of $X$ of size $k$. Also, $C_2$ is the chain poset of size 2, i.e. two distinct comparable elements, $\mathcal V$ is the 3-point poset comprised of a minimal element and two other incomparable elements above it, and $\Lambda$ is the 3-point poset comprised of a maximal element and 2 other incomparable elements below it. They are depicted below.
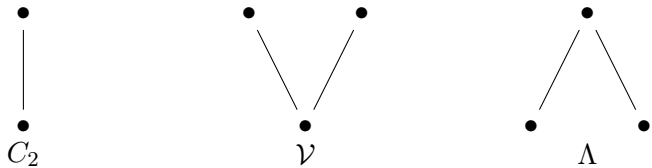
\begin{figure}[hbt!]
\centering
\begin{tikzpicture}[scale=0.75]
\node at (-5,0) (C2below){$\bullet$};
\node at (-5,2) (C2above){$\bullet$};
\draw (C2below) -- (C2above);
\node at (-5,-0.5) {$C_2$};
\node at (0,0)(Vbottom) {$\bullet$};
\node at (1,2)(Vright) {$\bullet$};
\node at (-1,2)(Vleft) {$\bullet$};
\draw (Vright) -- (Vbottom) -- (Vleft);
\node at (0,-0.5) {$\mathcal{V}$};
\node at (5,2)(Ltop) {$\bullet$};
\node at (6,0) (Lright){$\bullet$};
\node at (4,0)(Lleft) {$\bullet$};
\draw (Lright) -- (Ltop) -- (Lleft);
\node at (5,-0.5) {$\Lambda$};
\end{tikzpicture}
\end{figure}
\section{Setup and preliminary lemmas}
In this section we establish the main framework and notations, as well as some short lemmas that will be repeatedly used in the sections to follow.

Let $\mathcal F$ be a diamond-saturated family with ground set $[n]$. We recall a lemma proved in \cite{ferrara2017saturation}.
\begin{lemma}[Lemma 5 in \cite{ferrara2017saturation}]
Let $\mathcal F\subseteq\mathcal P([n])$ be a diamond-saturated family. If $\emptyset\in\mathcal F$, or $[n]\in\mathcal F$, then $|\mathcal F|\geq n+1$.
\end{lemma}
Therefore, since our aim is to show that $|\mathcal F|\geq n+1$, we may assume from now on that $\emptyset,[n]\notin\mathcal F$. Moreover, in this case, we can say something more, namely that there exists a minimal element incomparable to a maximal element. This was already proved in \cite{ivan2021minimal}, and below is the proof for completeness.
\begin{lemma}\label{AandBincomparable}Let $\mathcal F$ be a diamond-saturated family with ground set $[n]$ such that $\emptyset,[n]\notin\mathcal F$. Then there exist a minimal and a maximal element of $\mathcal F$ such that they are incomparable.
\end{lemma}
\begin{proof}
Indeed, suppose that every minimal element is comparable to every maximal element. 

Let $A_1, A_2, \dots, A_t$ be the minimal elements, and $B_1, B_2,\dots,B_l$ the maximal elements. Let  $X=\bigcup_{i=1}^t A_t$. By our assumption we have that $X\subseteq B_i$ for all $i\in[l]$. We now note that $[n]\setminus X$ is incomparable to $A_i$ and $B_j$, for all $i\in [t]$ and $j\in[l]$. This means that, not only it is not a member of $\mathcal F$, but it is incomparable to every set in $\mathcal F$. Consequently, this implies that it does not form a diamond added when added to $\mathcal F$, a contradiction.
\end{proof}

Next, let $\mathcal A$ be the set of minimal elements of $\mathcal F$ and $\mathcal B_0$ be the set of elements of $\mathcal P([n])$ below a copy of $\Lambda$ in $\mathcal F$. More precisely,
$$\mathcal B_0=\{X\in\mathcal P([n]):\exists P,R, Q\in\mathcal F\text { s.t. }P,Q,R,X\text{ form a $\mathcal D_2$ in which }X\text{ is the minimal element}\}.$$ 

Let $\mathcal B_1$ to be the set of maximal elements of $\mathcal B_0$. Finally, let $$\mathcal B = \{X \in \mathcal B_1 : X \text{ is not a subset of any set in }\mathcal A\}.$$

Additionally, let $G(\mathcal B)$ be all the sets in $\mathcal F$ that generate $\mathcal B$, in a certain sense. More precisely, for every $B\in\mathcal B$, let $G(B)=\{P\in\mathcal F:\exists R,Q\in\mathcal F\text{ such that } B,P,R,Q \text{ form a $\mathcal D_2$ s.t. } B\text{ is the minimal} \\\text{ element, and } P \text{ is one of the two middle elements}\}$. Then $G(\mathcal B)=\bigcup_{B\in\mathcal B}G(B)$. We also define $W=\{i\in[n]:i\notin X\text{ for all }X\in\mathcal A\}$. 

We also define the analogous sets $\mathcal X$, $\mathcal Y$ and $\mathcal Z$, which correspond exactly to the ones defined above, if one changes the diamond-saturated family from $\mathcal F$ to $\{[n]\setminus X:X\in\mathcal F\}$. Let $\mathcal X$ be the set of maximal elements of $\mathcal F$ and $\mathcal Y_0$ be the set of elements above a copy of $\mathcal V$ in $\mathcal F$.  More precisely,
$$\mathcal Y_0=\{X\in\mathcal P([n]):\exists P, Q, R \in\mathcal F\text { s.t. }P,Q,R,X\text{ form a $\mathcal D_2$ in which }X\text{ is the maximal element}\}.$$

Let $\mathcal Y_1$ be the set of minimal elements of $\mathcal Y_0$. Finally, let $$\mathcal Y=\{X \in \mathcal Y_1 : X \text{ does not contain any set of }\mathcal X \text{ as a subset}\}.$$

Additionally, let ${H}(\mathcal Y) =\bigcup_{Y\in\mathcal Y}H(Y)$, where $H(Y)=\{P\in\mathcal F:\exists R,Q\in\mathcal F\text{ such that } Y,P,R,Q \\ \text{ form a $\mathcal D_2$ s.t. } Y\text{ is the maximal element,}\text{ and } P \text{ is one of the two middle elements}\}$. Finally, let $\overline{W} = \{i\in [n] : i \in X \text{ for every } X \in \mathcal X\}$.

The first thing we note is the following property. The proof is identical to the proof of Lemma 3 in \cite{diamondlinear}, if one considers the diamond-saturated family $\{[n]\setminus X:X\in\mathcal F\}$.
\begin{lemma}\label{niceproperty1}
Let $\mathcal A$ and $\mathcal B$  be defined as above. Then $\mathcal A$ and $\mathcal B$ are disjoint, and $\mathcal A \cup \mathcal B$ is a $C_2$-saturated family of $\mathcal P([n])$.
\end{lemma}
In fact, if $B\subsetneq X$ for some $B\in\mathcal B$ and $X\in\mathcal P([n])$, then $X\in\mathcal F$, in which case $X$ contains a minimal element of $\mathcal F$, so an element of $\mathcal A$, or $X\notin\mathcal F$. If the latter is the case, then $X$ cannot be the minimal element of a diamond by the maximality of $B$, thus it contains an element of $\mathcal F$, and consequently an element of $\mathcal A$.

Therefore, for any $X\in\mathcal P([n])$, there exists $A\in\mathcal A$ such that $A\subseteq X$, or there exists $A'\in\mathcal A$ such that $X\subsetneq A$, or there exists $B\in\mathcal B$ such that $X\subseteq B$. This is how we will be applying Lemma~\ref{niceproperty1} in what follows.

The next lemma comes directly from the work in \cite{ivan2021minimal}.
\begin{lemma}\label{nicelemmame}
Suppose that $W\neq\emptyset$. Let $i\in W$ and $A\in\mathcal A$. Then there exists a set $S\in\mathcal F$ such that $A\subseteq S$ and $S\cup\{i\}\in\mathcal F$.
\end{lemma}
\begin{proof}
By the definition of $W$, $i\notin A$. If $A\cup\{i\}\in\mathcal F$, we are done. Therefore, we may assume that $A\cup\{i\}\notin\mathcal F$. Hence, there exist $3$ elements $B,C,D \in \mathcal{F}$ that form a diamond with $A\cup\{i\}$. First, $A\cup\{i\}$ cannot be the minimal element of the diamond as this would mean that $A,B,C,D$ is a diamond inside $\mathcal F$. Moreover, $A\cup\{i\}$ cannot be the maximal element of the diamond either, since, without loss of generality, we may assume that the minimal element of the diamond, say $D$, which has size  at most $|A|-1$, is in $\mathcal A$.  This implies that $i\notin D$ as $i\in W$, so thus $D\subsetneq A$, a contradiction. Therefore, $A\cup\{i\}$ has to be one of the middle elements, as depicted below.
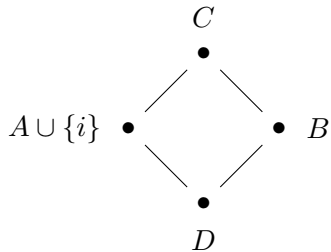
\begin{figure}[hbt!]
\hspace{0.1cm}\\
\centering
\begin{tikzpicture}
\node[label=above:$C$] (top) at (4,1) {$\bullet$}; 
\node[label=left:$A\cup\{i\}$] (left) at (3,0) {$\bullet$};
\node[label=right:$B$] (right) at (5,0) {$\bullet$};
\node[label=below:$D$] (bottom) at (4,-1) {$\bullet$};
\draw (top) -- (left) -- (bottom) -- (right) -- (top);
\end{tikzpicture}
\end{figure}
\FloatBarrier
Let $B$ be maximal with respect to this configuration. Without loss of generality, we may assume that $D$ is a minimal element of $\mathcal F$. This means that $i\notin D$, and so $D\subseteq A$. Since $D,A\in\mathcal A$, we must have that $D=A$. Consequently, this implies that $A\subset B$, and since $A\cup\{i\}$ and $B$ are incomparable, we have that $i\notin B$. If $B\cup\{i\}\in\mathcal F$, we are done. 

Therefore, we may assume that $B\cup\{i\}\notin\mathcal F$. Thus, there exist $X,Y,Z\in\mathcal F$ such that $X,Y,Z,B\cup\{i\}$ form a diamond. Since $A\subset\B\cup\{i\}\subset C$, $B\cup\{i\}$ cannot be the maximal or the minimal element of the diamond. Thus, we have the following diagram.

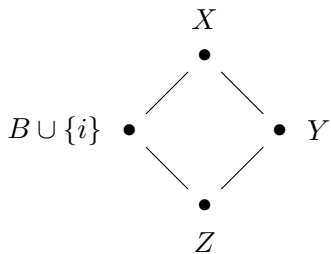
\begin{figure}[hbt!]
\hspace{0.1cm}\\
\centering
\begin{tikzpicture}
\node[label=above:$X$] (top) at (4,1) {$\bullet$}; 
\node[label=left:$B\cup\{i\}$] (left) at (3,0) {$\bullet$};
\node[label=right:$Y$] (right) at (5,0) {$\bullet$};
\node[label=below:$Z$] (bottom) at (4,-1) {$\bullet$};
\draw (top) -- (left) -- (bottom) -- (right) -- (top);
\end{tikzpicture}
\end{figure}
\FloatBarrier
Again, without loss of generality, we may assume that $Z\in\mathcal A$, hence $Z\subseteq B$. Since $D\subsetneq B$, $B\notin \mathcal{A}$, thus $Z\neq B$ and $Z\subsetneq B$. Since $B, Z, Y, X$ cannot form a diamond, $B$ and $Y$ must be comparable. If $Y\subseteq B$, then $Y\subset B\cup\{i\}$, a contradiction. Thus, $B\subsetneq Y$, which also implies that $i\notin Y$. This means that $A\cup\{i\}, A, Y, X$ form a diamond, contradicting the maximality of $B$.
\end{proof}
\begin{lemma}\label{sizeA} Let $A\in\mathcal A$. Then there exist at least $|A|$ elements $F\in\mathcal F$ such that $|F|\geq|A|$. Similarly, if $Y\in\mathcal Y$, there exist at least $n-|\mathcal Y|$ of size at most $|Y|$.
\end{lemma}
\begin{proof}
It is enough to find sets $A_i$ such that $A\setminus A_i=\{i\}$ and $|A_i|\geq|A|$ for all $i\in A$. To that effect, let $i\in A$. Since $A$ is a minimal element of $\mathcal F$, $A\setminus\{i\}\notin\mathcal F$. Therefore, $A\setminus\{i\}$ must form a diamond with 3 other elements of $\mathcal F$. Again, by minimality, $A\setminus\{i\}$ has to be the minimal element of the diamond. Let $C,D\in\mathcal F$ be the two incomparable elements of the diamond above it. To avoid a diamond being formed in $\mathcal F$, either $C$ or $D$ is $A$, or at least one of $C$ and $D$ does not contain $i$. If, without loss of generality, $C=A$, then $A$ and $D$ are incomparable, so $i\notin D$. Thus $A\setminus D=\{i\}$ and $|D|\geq|A|$. If $D\neq A$ and $C\neq A$, then, without loss of generality, $i\notin C$, so $C\setminus A=\{i\}$ and $|C|\geq |A|$.

The second part of the claim can be obtained by applying the above to the diamond-saturated family $\mathcal F^c=\{[n]\setminus F:F\in\mathcal F\}$.
\end{proof}
\begin{lemma}\label{sizeB}Let $B\in\mathcal B$. For every $i\notin B$, there exists $X_i\in\mathcal F$ such that $X_i\subseteq B\cup\{i\}$ and $i\in B_i$. Similarly, given $C\in\mathcal Y$, for every $i\in C$, there exists $Y_i\in\mathcal F$ such that $C\setminus\{i\}\subseteq Y_i$ and $i\notin Y_i$.
\end{lemma}
\begin{proof}
Let $i\notin B$. If $B\cup\{i\}\in\mathcal F$, then we are done. Thus, we may assume that $B\cup\{i\}\notin\mathcal F$, and so it forms a diamond with 3 other elements of $\mathcal F$. We note that $B\cup\{i\}$ cannot be the minimal element of the diamond. Indeed, if it is, by the maximality of $B$, it cannot be in $\mathcal B$. Thus, there exists $A\in\mathcal A$ such that $B\cup\{i\}\subset A$, which implies that $B\subset A$, a contradiction. 

Therefore, there exists $X\in\mathcal F$ such that $X\subset B\cup\{i\}$. We now recall that, since $B\in\mathcal B$, there exist elements $C,D,E\in\mathcal F$ such that $B,C,D,E$ form a diamond, where $B$ is the minimal element. If $X\subseteq B$, then $X,C,D,E$ form a diamond in $\mathcal F$, a contradiction. Thus $X\subset B\cup\{i\}$ and $i\in X$, which finishes the first part of the lemma. The second part follows immediately by taking complements.
\end{proof}
\begin{lemma}\label{G(B) and A are disjoint} Let $G(\mathcal B),\mathcal A, H(\mathcal Y), \mathcal X$ be as defined above. Then $G(\mathcal B)\cap \mathcal A =\emptyset$, and $H(\mathcal Y) \cap \mathcal X= \emptyset$.
\end{lemma}
\begin{proof}
Suppose that there exists $A\in G(\mathcal B)\cap \mathcal A$. Since $A\in G(\mathcal B)$, this means that there exists $B\in\mathcal B$ such that $A\in G(B)$, which implies that $B\subset A$, as $A$ and $B$ are part of a diamond in which $B$ is the minimal element. However, this is a direct contradiction to the definition of $\mathcal B$. A completely analogous proof gives that $H(\mathcal Y) \cap \mathcal X= \emptyset$. 
\end{proof}
In what follows, we will make use of the following standard lemma, whose simple proof we omit.
\begin{lemma}\label{appendixproposition} Let $k,h,n\in \mathbb{N}$, and $A_1,\dots,A_k$ be non-empty subsets of $[n]$. Suppose that for every $B\in \binom{[n]}{h}$, there exists $i\in[k]$ such that $A_i\subseteq B$. Then $k\geq n - h + 1$.
\end{lemma}

\section{Lower bounding the size of $\mathcal A\cup G(\mathcal B)$}

We begin with some notation. Let $\mathcal G\subseteq \cP([n])$ be a family of sets. For every $i\in [n]$ we define $f(i,\mathcal G)$ to be the family of sets in $\mathcal G$ that contain $i$, namely $\{G\in \mathcal G: i\in G\}$. We also define $W(\mathcal G)$ to be the set of elements that appear in no set of $\mathcal G$, namely $\{i\in [n] : f(i,\mathcal G) = \emptyset\}$.

A crucial observation is that, given $W'\subset[n]\backslash W(\mathcal G)$, the collection of non-empty families $\{f(i,\mathcal G) : i\in W' \}$ is partially ordered by inclusion.

Returning to our diamond-saturated family and the setup defined in the previous section, if we take $\mathcal G=\mathcal A$, we have that $W = W(\mathcal A)$.

We now construct a sequence of  distinct singletons $a_1,\dots,a_k\in [n]$  and nested families $\mathcal A_1\supsetneq\mathcal A_2\supsetneq\dots\supsetneq\mathcal A_k\neq\emptyset $ as follows. Let $\mathcal A = \mathcal A_1$. Clearly $[n]\setminus W(\mathcal A_1)\neq\emptyset$, so we may pick $a_1$ such that $f(a_1,\mathcal A_1)$ is an inclusion-minimal element of $\{ f(i,\mathcal A_1) : i\in [n]\setminus W(\mathcal A_1)\}$. Next, let $\mathcal A_2 = \mathcal A_1\setminus f(a_1,\mathcal A_1)$. For each $i\geq 2$, if $[n]\setminus W(\mathcal A_i)$ is not empty, pick $a_i$ such that $f({a_i},\mathcal A_i)$ is an inclusion-minimal element of $\{ f(j,\mathcal A_i) : j\in [n]\setminus W(\mathcal A_i)\}$. If, on the other hand, $[n]\setminus W(\mathcal A_i)= \emptyset$ (which is equivalent to $\mathcal A_i=\emptyset$), we terminate the sequence at $i-1$, thus making $k = i-1$.

This way, we obtain a sequence $a_1,\dots,a_k$ of distinct singletons in $[n]$, and a strictly nested sequence of families $\mathcal A=\mathcal A_1\supsetneq\mathcal A_2\supsetneq\dots\supsetneq\mathcal A_k\neq\emptyset$. For all $1\leq i\leq k$, we define $A_i$ to be the set of singletons that cannot be separated from $a_i$ in $\mathcal A_i$, or more formally,  $A_i = \{j\in [n] : f(j,\mathcal A_i) = f(a_i,\mathcal A_i)\}$. We notice that these sets are pairwise disjoint. Indeed, if  $z\in A_i\cap A_j$ for some $i<j$, then $f(z,\mathcal A_i)=f(a_i,\mathcal A_i)$, and $\mathcal A_j\subseteq\mathcal A_i\setminus f(z,\mathcal A_i)$, which implies that no set in $\mathcal A_j$ contains $z$, a contradiction. Therefore $A_i\cap A_j=\emptyset$ for all $i\neq j$.

In what follows, we will construct $k$ distinct elements of $\mathcal A$, and, under certain assumptions, $\sum_{i=1}^k (|A_i|-1)$ distinct elements of $G(\mathcal B)$. Since $\mathcal A$ and $G(\mathcal B)$ are disjoint, we obtain a lower bound on $|\mathcal A\cup G(\mathcal B)|$. We start with a couple of short helpful lemmas. The first tells us that not only the $A_i$ are disjoint, but they actually partition the `arena', i.e., $[n]\setminus W$.


\begin{lemma}\label{proposition on union of A_i equals [n] backslash W}
For the sets defined above, we have that $\bigcup_{i=1}^k A_i=[n]\setminus W$.
\end{lemma}
\begin{proof}
One direction is clear, as for all $1\leq i\leq k$, $A_i\subseteq [n]\setminus W$, by construction. In the other direction, let $x\in [n]\setminus W$. By definition, there exists $X\in\mathcal A=\mathcal A_1$ such that $x\in X$, thus $x\notin W(\mathcal A_1)$. Let $j$ be the largest $j'\in[k]$ such that $x\notin W(\mathcal A_{j'})$, which also gives that $f(x,\mathcal A_j)\neq\emptyset$. If $x\not\in A_j$, then $f(a_j,\mathcal A_j)\neq f(x,\mathcal A_j)$. Hence, by the inclusion-minimality of $f(a_j,\mathcal A_j)$, there exists a set $D\in f(x,\mathcal A_j)\setminus f(a_j,\mathcal A_j)\subseteq \mathcal A_j\setminus f(a_j,\mathcal A_j)$. This implies that $j\neq k$, and that $D\in \mathcal A_{j+1}$. Thus $x\notin W(\mathcal A_{j+1})$, contradicting the maximality of $j$. We therefore have that $x\in A_j$, and consequently, $[n]\setminus W\subseteq\bigcup_{i=1}^k A_i$, which finishes the proof.
\end{proof}

\begin{lemma}\label{Proposition on property of elements of mathcal A_j } Let $1\leq j\leq k$ and $T\in\mathcal A_j$. Then $A_l\cap T = \emptyset$ for all $l\leq j-1$.
\end{lemma}
\begin{proof}
The proof is by induction on $j$. If $j=1$, the claim is vacuously true. Suppose now that $j>1$, and that the claim holds for $j-1$. First, since $\mathcal A_j = \mathcal A_{j-1}\setminus f(a_{j-1},\mathcal A_{j-1})\subset\mathcal A_{j-1}$, we have by the induction hypothesis that $A_l\cap T = \emptyset$ for all $1\leq l\leq j-2$.

Next, suppose that there exists $x\in T\cap A_{j-1}$. Since $T\in\mathcal A_j\subset A_{j-1}$, we have that $T\in f(x,\mathcal A_{j-1})$. However, since $x\in A_{j-1}$, we also have that $f(x,\mathcal A_{j-1}) = f(a_{j-1},\mathcal A_{j-1})$, hence $T\in f(a_{j-1},\cA_{j-1})$. By construction, this means that $T\not\in \mathcal A_j$, a contradiction. Therefore $T\cap A_{j-1} = \emptyset$, which finishes the induction, and consequently the proof. 
\end{proof}
We are now ready to generate $k$ distinct elements of $\mathcal A$.
\begin{proposition}\label{Result on the elements of mathcal A}
Let $1\leq i\leq k$. Then there exists $X_i\in \mathcal A$ such that $X_i\cap \left(A_1\cup A_2\cup\dots\cup A_i\right) = A_i$.
\end{proposition}
\begin{proof}
By construction, $f(a_i,\mathcal A_i)$ is not empty, so there exists $X_i\in f(a_i,\cA_i)$.  Since $f(a_i,\cA_i)\subseteq \mathcal A_{i}$, by Lemma~\ref{Proposition on property of elements of mathcal A_j }, $X_i\cap A_j = \emptyset$ for all $j\leq i-1$. Furthermore, since $X_i\in f(a_i,\mathcal A_i)$, we have that $X_i \in f(x,\mathcal A_i)$ for all $x\in A_i$, which implies that $x\in X_i$ for all $x\in A_i$. Therefore $A_i\subseteq X_i$. Putting everything together, we indeed have that $X_i\cap \left(A_1\cup A_2\cup\dots\cup A_i\right) = A_i$, as claimed.
\end{proof}

We now show that, under certain assumptions, we can generate $\sum_{i=1}^k(|A_i|-1)$ distinct sets in $G(\mathcal B)$. This is one of the main building blocks for reaching a lower bound on $|\mathcal A\cup G(\mathcal B)|$, which we do at the end of this section. 

Let us denote by $m_{\mathcal A}$ the maximal size of a set in $\mathcal A$. In other words, $m_{\mathcal A} = \max_{A\in \mathcal A} |A|$.

\begin{proposition}\label{Main Lemma of second section with no assumptions on W} Suppose that $|\mathcal A\cup G(\mathcal B)|\leq n-m_{\mathcal A}$. 

Then, given $1\leq j\leq k$ such that $|A_j|>1$, and $x\in A_j\setminus\{a_j\}$, there exists $B_x\in  G(\cB)$ such that $B_x \cap A_j = A_j\setminus\{x\}$, and $A_l\backslash \{a_l\}\subseteq B_x$ for all $l\leq j-1$.

Furthermore, there exists $C_x\in G(\mathcal B)$ such that $C_x\neq B_x$, $A_j\setminus\{x\}\subseteq C_x$, and $A_l\setminus\{a_l\}\subseteq C_x$ for all $l\leq j-1$.
\end{proposition}
Before we dive into the proof, the plan is as follows: choose a certain set $S$ of $f(a_j,\mathcal A_j)$ and remove $x$. Ideally, we would want to look at $S\setminus\{x\}$ and use the fact that $\mathcal A\cup\mathcal B$ is $C_2$-saturated to obtain that it is a subset of a set $B\in\mathcal B$. We then use the diamond $B$ comes with to generate the 2 incomparable sets which will be our $B_x$ and $C_x$. However, this exact approach may not always work, hence we add to $S\setminus\{x\}$ a suitable set $\overline{Y}$ to enforce this.
\begin{proof} The proof is by induction on $j$. We first perform the induction step, as the base case is highly similar. Suppose that $j>1$ and that the statement holds for $1,2,\dots,j-1$. We may assume that $A_j\setminus \{a_j\}\neq\emptyset$.

By construction, $f(a_j,\mathcal A_j)$ is not empty, so let $S$ be a set of $ f(a_j,\mathcal A_j)$ of maximal cardinality. We note that, by construction, $A_j\subseteq S$. Indeed, if $y\in A_j$, then $S\in f(a_j,\mathcal A_j)=f(y,\mathcal A_j)$, and so $y\in S$.

\begin{claim888}$|\mathcal A\cup G(\mathcal B)|\geq \left| A_1\cup A_2\cup\dots\cup A_{j-1}\right|+1$.
\end{claim888}
\begin{proof} 
The proof of this claim uses the sets in $\mathcal A$ generated by Proposition~\ref{Result on the elements of mathcal A}, and the inductive hypothesis.
By Proposition~\ref{Result on the elements of mathcal A}, there exists an $X_i\in \mathcal A$ such that $X_i\cap \left(A_1\cup A_2\cup\dots\cup A_i\right)= A_i$, for all $i\in[j]$. Let $\mathcal F_{\mathcal A} = \{X_i : 1\leq i\leq j\}\subseteq\mathcal A$. If $i<i'$, then $X_{i'}$ has an empty intersection with $A_1\cup\dots \cup A_i$, and so $X_i\neq X_{i'}$. Therefore, we have that $|\mathcal F_{\mathcal A}|=j$. 

Moreover, by our inductive hypothesis, for every $1\leq t\leq j-1$ for which $|A_t|>1$, and $y\in A_t\setminus\{a_t\}$, there exists a set $B_y\in G(\mathcal B)$ such that $B_y \cap A_t = A_t\setminus \{y\}$, and $A_l\setminus\{a_l\}\subseteq B_y$ for all $1\leq l \leq t-1$. This gives a set $B_y$ for each $y\in \bigcup_{i=1}^{j-1} \left(A_i\setminus \{a_i\}\right)$. 

Suppose that $B_y = B_z$ for some $y,z \in \bigcup_{i=1}^{j-1} \left(A_i\setminus\{a_i\}\right)$. Then $y\in A_i\setminus \{a_i\}$ and $z\in A_{i'}\setminus\{a_{i'}\}$ for some $i,i'\in[j-1]$. Without loss of generality, we may assume that $i\leq i'$. If $i<i'$, then $A_i\setminus \{a_{i}\}\subseteq B_z$, which implies that $y\in B_z=B_y$, a contradiction. Therefore $i=i'$. Thus $A_i \setminus\{y\} = B_y \cap A_{i} = B_z\cap A_i = A_i \setminus\{z\}$, which implies that $y=z$. Therefore, the sets $B_y$ are pairwise different. Let $\mathcal F_{\mathcal B}= \{B_y : y\in\bigcup_{i=1}^{j-1} \left(A_i\setminus \{a_i\}\right)\}\subseteq G(\mathcal B)$. The above tells us that $|\mathcal F_{\mathcal B}|= |\bigcup_{i=1}^{j-1} (A_i\setminus \{a_i\})|=|\bigcup_{i=1}^{j-1} A_i | - j +1$.

By Lemma~\ref{G(B) and A are disjoint} we have that $G(\mathcal B)$ and $\mathcal A$ are disjoint, and thus $\mathcal F_{\mathcal B}$ and $\mathcal F _{\mathcal A}$ are also disjoint. Furthermore, $|\mathcal A\cup \G(\mathcal B)|\geq|\mathcal F_{\mathcal A}|+ |\mathcal F_{\mathcal B}|=j+|\bigcup_{i=1}^{j-1} A_i | - j +1$. Therefore, we have that $|\mathcal A\cup G(\mathcal B)|\geq\left| A_1\cup A_2\cup\dots\cup A_{j-1}\right|+1$, as claimed. 
\end{proof}
We are now ready to move to the main part of the induction, so let $x\in A_j\setminus\{a_j\}$. We define $S^* = S\cup \left( \bigcup_{i=1}^{j-1} (A_{i}\setminus\{a_j\})\right)$. 

We note that $|[n]\setminus\left( S\cup A_1\cup A_2\cup\dots\cup A_{j-1}\right) |\geq n - |S| - |A_1\cup A_2\cup\dots\cup A_{j-1}| $, which by the above claim is greater than or equal to $n+1 - |S| - |\mathcal A\cup G(\mathcal B)|$, which is at least $m_{\mathcal A}-|S|+1$ by the initial assumption. For simplicity, we denote by $\mathcal T$ the family of sets $\binom{[n]\setminus \left( S\cup A_1\cup A_2\cup\dots\cup A_{j-1}\right)}{m_A - |S|+1}$.
 
Let $Y\in\mathcal T$, and consider $(S^{^*}\setminus\{x\})\cup Y$. By Lemma~\ref{niceproperty1}, there exists a set $A\in \mathcal A$ such that $A\subseteq (S^*\setminus\{x\})\cup Y$, or a set $A'\in \mathcal A$ such that $(S^*\setminus\{x\})\cup Y \subsetneq A'$, or a set $B\in \mathcal B$ such that $(S^*\setminus\{x\})\cup Y\subseteq B$.

\begin{claim888} There exists a set $Y\in\mathcal T$ such that $A\not\subseteq (S^*\setminus \{x\})\cup Y$ for all $A\in\mathcal A$.
\end{claim888}
\begin{proof}
Suppose that for all $Y\in\mathcal T$, there exists a set $A\in \mathcal A$ such that $A\subseteq (S^*\setminus\{x\})\cup Y$. 

First we show that this implies that $A\in \mathcal A_l$ for all $1\leq l\leq j$. We prove this by induction on $l$. The base case $l= 1$ is trivial by construction as $A\in\mathcal A=\mathcal A_1$. Suppose now that $2\leq l\leq j$ and $A\in \mathcal A_{l-1}$. We observe that $S^*=S\cup \left( \bigcup_{i=1}^{j-1} (A_{i}\setminus\{a_i\})\right)$, and since $a_{l-1}\notin S$ as $S\in\mathcal A_j$, $a_{l-1}\not \in S^*$. Furthermore, by construction, $a_{l-1}\notin Y$. We therefore have that $a_{l-1}\not\in A$, which together with the fact that $A\in\mathcal A_{l-1}$, implies that $A\in \mathcal A_{l}$, completing the inductive step.

Therefore we have that $A\in\mathcal A_j$. Combining this with Lemma~\ref{Proposition on property of elements of mathcal A_j }, we have that $A\cap \left(\bigcup_{i=1}^{j-1} A_i\right) = \emptyset$. As such, since $A\subseteq (S^*\setminus\{x\} )\cup Y$, we must in fact have $A\subseteq (S\setminus\{x\})\cup Y$. Thus $A\setminus (S\setminus\{x\})\subseteq Y$. Furthermore, $A\setminus (S\setminus\{x\})$ is not empty, as otherwise $A$ would be a subset of $S\backslash \{x\}$, hence a strict subset of $S\in\mathcal A$, contradicting the fact that $\mathcal A$ is an antichain. 

Thus, for every $Y\in \mathcal T$, there exists a set $A_Y\in \mathcal A_j\subseteq\mathcal A$ such that $ A_Y\setminus (S\setminus \{x\}) \subseteq Y$ and $A_Y\setminus (S\setminus \{x\})\neq \emptyset$. 

We have that $\left|\left\{A_Y : Y \in \mathcal T\right\}\right|\geq \left|\left\{A_Y\setminus(S\setminus\{x\}) : Y \in \mathcal T\right\}\right|$. By Lemma~\ref{appendixproposition}, we also have that $$\left|\left\{A_Y\setminus(S\setminus\{x\}) : Y \in \mathcal T\right\}\right|\geq\left|[n]\setminus\left( S\cup A_1\cup\dots\cup A_{j-1}\right)\right|- m_{\mathcal A} +|S|,$$ which is at least $n - |\bigcup_{i=1}^{j-1} A_i| -  m_{\mathcal A}$. If we denote by $\mathcal F'_{\mathcal A}$ the set $\left\{A_Y : Y \in \mathcal T\right\}$, we have that $\mathcal F'_{\mathcal A}\subseteq\mathcal A_j\subseteq\mathcal A$ and $|\mathcal F'_{\mathcal A}|\geq n - |\bigcup_{i=1}^{j-1} A_i| -  m_{\mathcal A}$.

We recall that in the proof of Claim 1, we exhibited two families $\mathcal F_{\mathcal A}\subseteq \mathcal A$ and $\mathcal F_{\mathcal B} \subseteq G(\mathcal B)$ whose union has size $|\bigcup_{i=1}^{j-1} A_i|+1$. We will show that these families are disjoint from the family $\mathcal F'_{\mathcal A}$ constructed above. This will then give at least $n+1-m_{\mathcal A}$ elements of $\mathcal A\cup G(\mathcal B)$, which will contradict the initial assumption.

Firstly, since $\mathcal F'_{\mathcal A}\subseteq\mathcal A$, $\mathcal F_{\mathcal B}\subseteq G(\mathcal B)$, and $\mathcal A$ and $G(\mathcal B)$ are disjoint, we have that $\mathcal F'_{\mathcal A}\cap\mathcal F_{\mathcal B}=\emptyset$.

Next, suppose that there exists $P\in\mathcal F_{\mathcal A}\cap\mathcal F'_{\mathcal A}$. Since $P\in\mathcal F_{\mathcal A}$, there exists $i\in[j]$ such that $P\cap (\bigcup_{l=1}^{i}  A_l) =A_i$. Since also $P\in \mathcal A_{j}$, by Lemma~\ref{Proposition on property of elements of mathcal A_j }, we must therefore have $i=j$. 

However, $P\subseteq (S\setminus\{x\} )\cup Y$ for some $Y\in \mathcal T$, which implies that $x\not\in P$. Thus, since $x\in A_j$, which implies that $f(a_j,\mathcal A_j) = f(x,\mathcal A_j)$, we have that $a_j\not\in P$. Therefore $P\in \mathcal A_j\setminus f(a_j,\mathcal A_j) = \mathcal A_{j+1}$. Combining this with Lemma~\ref{Proposition on property of elements of mathcal A_j }, we get that $P\cap \left(\bigcup_{l=1}^{j} A_i\right) = \emptyset$, a contradiction.
\end{proof}

The above guarantees the existence of a set $\overline{Y}\in \mathcal T$ for which there exists a set $A\in \mathcal A$ such that $(S^*\setminus\{x\})\cup\overline{ Y} \subsetneq A$, or there exists a set $B\in \mathcal B$ such that $(S^*\setminus\{x\})\cup\overline{ Y} \subseteq B$.
\begin{claim888}
Let $Y\in \mathcal T$, $A\in \mathcal A$. Then $(S^*\setminus\{x\})\cup Y$ is not a strict subset of $A$.
\end{claim888}
\begin{proof}
Suppose that $(S^*\setminus\{x\})\cup Y \subsetneq A$. Then $|A|> |(S^*\setminus\{x\})\cup Y |\geq |(S\setminus\{x\})\cup Y|$. Since $Y$ and $S$ are disjoint, this is equivalent to $|A|> |S|-1+|Y| = |S|-1+m_{\mathcal A}-|S|+1=m_{\mathcal A}$, a contradiction.
\end{proof}
Therefore, there exists $B\in\mathcal B$ such that $(S^*\setminus\{x\})\cup\overline{Y} \subseteq B$, which implies that $S\setminus\{x\}\subseteq B$. Since $S\in\mathcal A$, Lemma~\ref{niceproperty1} tells us that $S\not\subseteq B$ , thus $x\not\in B$. Recall that $A_j\subseteq S$, so $A_j\setminus \{x\}\subseteq B$. This means that $B\cap A_j = A_j\setminus\{x\}$. Furthermore, since $S^*\setminus\{x\}\subseteq B$, we have that $A_l\setminus \{a_l\}\subseteq B$ for all $1\leq l\leq j-1$. 

By definition, since $B\in\mathcal B$, there exist sets $C,D,E\in \mathcal F$ such that $B, C, D, E$ forms a diamond with $B$ as its minimal element, as illustrated below.
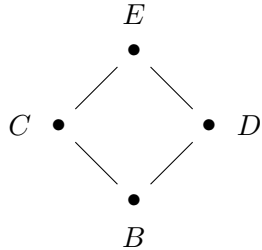
\begin{figure}[hbt!]
\hspace{0.1cm}\\
\centering
\begin{tikzpicture}
\node[label=above:$E$] (top) at (4,1) {$\bullet$}; 
\node[label=left:$C$] (left) at (3,0) {$\bullet$};
\node[label=right:$D$] (right) at (5,0) {$\bullet$};
\node[label=below:$B$] (bottom) at (4,-1) {$\bullet$};
\draw (top) -- (left) -- (bottom) -- (right) -- (top);
\end{tikzpicture}
\end{figure}
\FloatBarrier
By definition, $C$ and $D$ are elements of $G(\mathcal B)$ and, furthermore, $B= C\cap D$, by the maximality of $B$. Combining this with the previous observations we have that $C\cap D\cap A_j = A_j\setminus \{x\}$ and $A_l\setminus \{a_l\}\subseteq C\cap D$ for all $1\leq l\leq j-1$. The first equation implies that one of $C$ or $D$ does not contain $x$, hence its intersection with $A_j$ must be $A_j\setminus\{x\}$. Without loss of generality, we may assume that $C\cap A_j = A_j\setminus \{x\}$. Then $C$ satisfies all the required conditions stated in the first part of the proposition. 

Additionally, $D$ satisfies $A_j\setminus\{x\}\subseteq D$ and $A_l\setminus\{a_l\}\subseteq D$ for all $1\leq l\leq j-1$. This completes the induction step, and hence the proof.
\end{proof}
Therefore, if $\mathcal A\cup G(\mathcal B)$ has size at most $n-m_{\mathcal A}$, then for every $x\in \bigcup_{l=1}^k (A_l\setminus \{a_l\})$ Proposition~\ref{Main Lemma of second section with no assumptions on W} produces a distinct element $B_x\in G(\mathcal B)$, giving that $|G(\mathcal B)|\geq \sum_{l=1}^k |A_k| - k$, which is equal to $n-|W|$ by Lemma~\ref{proposition on union of A_i equals [n] backslash W}. Therefore, coupled with Proposition~\ref{Result on the elements of mathcal A}, if $|\mathcal A\cup G(\mathcal B)|\leq n-m_{\mathcal A}$, then $|\mathcal A\cup G(\mathcal B)|\geq n-|W|$. This implies that, in general, $|A\cup G(\cB)|\geq n - m_{\mathcal A}- |W|$. However, this can be further improved to $|\mathcal A\cup G(\mathcal B)|\geq n +1 - m_{\mathcal A} - |W|$, which we do in the following corollary.

\begin{corollary}\label{Corollary bounding size of A cup G(B) when W empty}
We have that $|\mathcal A\cup G(\mathcal B)|\geq n+1-m_{\mathcal A}-|W|$. By symmetry, we also have that $|\mathcal X\cup H(\mathcal Y)|\geq n+1-\max_{X\in\mathcal X}(n-|X|)- |\overline{W}|$.
\end{corollary}
\begin{proof}
If $|\mathcal A\cup G(\mathcal B)|\geq n+1-m_{\mathcal A}$ we are done. Thus we may assume that $|\mathcal A\cup G(\mathcal B)|\leq n - m_{\mathcal A}$. Hence, by Proposition~\ref{Main Lemma of second section with no assumptions on W} we obtain $\sum_{i=1}^k (|A_i|-1)$ elements in $G(\mathcal B)$. We will prove that in fact $|G(\mathcal B)|\geq \sum_{i=1}^k (|A_i|-1) +1$. 

\begin{case1} $|A_l| = 1$ for all $1\leq l \leq k$.
\end{case1}
In this case, it is enough to prove that $\mathcal B$ is not empty. Suppose that $\mathcal B = \emptyset$. 

Let $S$ be a maximum cardinality element of $\mathcal A$ and let $i\in S$. Consider the sets $(S\setminus\{i\})\cup\{j\}$ for all $j\in[n]\setminus S$. By Lemma~\ref{niceproperty1}, for every $j\in [n]\setminus S$, there exists an element  $Y_j\in \mathcal A$ such that $A_j$ and $(S\setminus\{i\})\cup \{j\}$ are comparable. By the maximality of $|S|$, we must have $Y_j\subseteq (S\setminus \{i\})\cup\{j\}$. Since $\mathcal A$ is an antichain, we have that $A_j\not\subseteq S$, which means that $j\in Y_j$. Thus $Y_j\setminus S = \{j\}$. This also means that $Y_j\neq Y_{j'}$ if $j\neq j'$. Therefore $\{Y_j : j\in [n]\setminus S\} \cup \{S\}$ is a subset of $\mathcal A$ of size $n + 1 - |S|$. Thus, $|\mathcal A\cup G(\mathcal B)|\geq n+1 -|S|=n+1-m_{\mathcal A}$, a contradiction. Therefore, $\mathcal B\neq \emptyset$, so $|G(\mathcal B)|\geq1= \sum_{i=1}^k (|A_i|-1) +1$.

\begin{case1} $|A_l|\neq 1$ for some $l\in \{1,2,\dots,k\}$.
\end{case1}
Let $l\in [k]$ be the largest value such that $|A_l|\geq 2$, and let $x\in A_l\setminus\{a_l\}$. We will prove that the second set guaranteed by Proposition~\ref{Main Lemma of second section with no assumptions on W}, $C_x$, is distinct from all $B_y$ for all $y\in \bigcup_{i=1}^k (A_i \setminus\{a_i\})$. 

Suppose that  $C_x = B_y$ for some $y\in  \bigcup_{i=1}^k (A_i \setminus \{a_i\})$. Then $y\in A_i \setminus\{a_i\}$ for some $i\leq l$, as $A_{i'}\setminus\{a_{i'}\} = \emptyset $ for all $i' > l$. 

If $i< l$ then by Proposition~\ref{Main Lemma of second section with no assumptions on W} we have that $B_y\cap A_i=A_i\setminus\{y\}$, and $A_i\setminus\{a_i\}\subseteq C_x\cap A_i$. This implies that $A_i\setminus\{a_i\}\subseteq A_i\setminus\{y\}$, a contradiction, hence $i = l$. Therefore $ A_l\setminus \{y\} = B_y \cap A_l = C_x\cap A_l$. However, since $A_l\setminus \{x\}\subseteq C_x$, we get that $A_l\setminus\{x\}\subseteq A_l\setminus\{y\}$, so $y = x$. This is a contradiction, as $B_x\neq C_x$. Thus, by taking $C_x$ along with $\{B_y : y\in \bigcup_{j=1}^k (A_j\setminus\{a_j\})\}$, we obtain $\sum_{i=1}^k (|A_i|-1) +1$ elements in $G(\mathcal B)$.

Therefore, in both cases, we get that $|G(\mathcal B)|\geq \sum_{i=1}^k (|A_i|-1) +1$. By Proposition~\ref{Result on the elements of mathcal A}, $|\mathcal A|\geq k$. Combining these two inequalities gives $|\mathcal A\cup G(\mathcal B)|\geq\left|\bigcup_{i=1}^kA_k\right|+1=n+1-|W|$, which completes the proof.
\end{proof}
It turns out that when $W\neq \emptyset$, Corollary~\ref{Corollary bounding size of A cup G(B) when W empty} can be strengthened to $|\mathcal A\cup G(\mathcal B)|\geq n+1  - |W|$. The proof is similar to the proof of Proposition~\ref{Main Lemma of second section with no assumptions on W}.
\begin{proposition}\label{Lemma bounding A cup G(B) assuming W not empty} Suppose that $W\neq \emptyset$. Then, for all $1\leq j\leq k$ for which $|A_j|>1$, and all $x\in A_j\setminus\{a_j\}$, there exists $B\in  G(\mathcal B)$ such that $B \cap A_j = A_j\setminus\{x\}$, $A_l\setminus \{a_l\}\subseteq B$ for all $1\leq l \leq j-1$, and $W\subseteq B$. Furthermore, there exists $C\in G(\mathcal B)$, $C\neq B$, such that $W\subseteq C$, $A_j\setminus\{x\}\subseteq C$, and $A_l\setminus\{a_l\}\subseteq C$ for all $1\leq l\leq j-1$.
\end{proposition}
\begin{proof}
Let $1\leq j\leq k$ be such that $A_j\setminus\{a_j\}\neq\emptyset$, and $x\in A_j\backslash \{a_j\}$.

By construction, $f(a_j,\mathcal A_j)$ is not empty. Let $S\in f(a_j,\mathcal A_j)$, which implies that $x\in A_j\subseteq S$. 

We now define $S^* = S\cup \left( \bigcup_{l=1}^{j-1} (A_{l}\setminus \{a_l\})\right)$. By~Lemma \ref{niceproperty1} there exists $A\in \mathcal A$ such that $A\subseteq (S^*\setminus\{x\})\cup W$, or there exists  $A'\in \mathcal A$ such that $(S^*\setminus \{x\})\cup W \subsetneq A'$, or there exists $B'\in \mathcal B$ such that $(S^*\setminus\{x\})\cup W \subseteq B'$.

First, suppose that there exists $A\in \mathcal A$ such that $A\subseteq (S^*\setminus \{x\})\cup W$. The same way as in the first part of Claim 2 in the proof of Proposition~\ref{Main Lemma of second section with no assumptions on W}, we get that $A\in \mathcal A_j$, as the only thing that matters is that $a_i\notin A$ for all $i\in[j-1]$. Thus, by Lemma~\ref{Proposition on property of elements of mathcal A_j }, $A\cap \left(\bigcup_{i=1}^{j-1} A_i \right) = \emptyset$. Therefore, since $A\subseteq (S^*\setminus\{x\} )\cup W$, we in fact have that $A\subseteq (S\setminus \{x\})\cup W$. As by construction $A\cap W=\emptyset$ we have that $A\subseteq S\setminus\{x\}$. Hence $A$ is a proper subset of $S$, contradicting the fact that $\mathcal A$ is an antichain.

Furthermore, if $A'\in \mathcal A$ is such that $(S^*\setminus\{x\})\cup W \subsetneq A'$, then $W\subseteq A'$, a contradiction.

Therefore, there exists $B'\in\mathcal B$ such that $(S^*\setminus\{x\})\cup W \subseteq B'$. By Lemma~\ref{niceproperty1}, $B'$ and $S$ are incomparable. Since $S\setminus\{x\}\subseteq B'$, we have that $B'\cap A_j = A_j\setminus\{x\}$. Moreover, since $(S\setminus\{x\})\cup \left( \bigcup_{l=1}^{j-1} (A_{l}\setminus \{a_l\}\right)=S^*\setminus\{x\}\subseteq B'$, we have that $A_l\setminus\{a_l\}\subseteq B'$ for all $1\leq l\leq j-1$. 

By the definition of $\mathcal B$, there exist sets $B,C,D\in \mathcal F$ which, together with $B'$, form a diamond in which $B'$ is the minimal element, as depicted below.
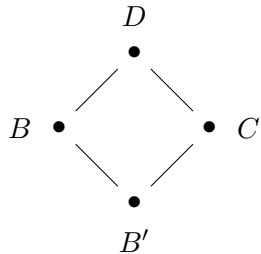
\begin{figure}[hbt!]
\hspace{0.1cm}\\
\centering
\begin{tikzpicture}
\node[label=above:$D$] (top) at (4,1) {$\bullet$}; 
\node[label=left:$B$] (left) at (3,0) {$\bullet$};
\node[label=right:$C$] (right) at (5,0) {$\bullet$};
\node[label=below:$B'$] (bottom) at (4,-1) {$\bullet$};
\draw (top) -- (left) -- (bottom) -- (right) -- (top);
\end{tikzpicture}
\end{figure}
\FloatBarrier
By definition, $B$ and $C$ are elements of $G(\mathcal B)$. Furthermore, by the maximality of $B'$, we have that $B' = B\cap C$. This implies that $B\cap C\cap A_j = A_j\setminus \{x\}$, and $A_l\setminus\{a_l\}\subseteq B\cap C$ for all $1\leq l\leq j-1$. The first equation implies that $B\cap A_j = A_j\setminus\{x\}$, or $C\cap A_j =A_j\setminus \{x\}$. Without loss of generality, we may assume that $B\cap A_j = A_j\setminus\{x\}$. Trivially, since $W\subseteq B'$, we also have that $W\subseteq B$ and $W\subseteq C$. Then $B$ and $C$ satisfy all the required conditions, and hence the proof is complete.
\end{proof}
Following the blueprint from Corollary~\ref{Corollary bounding size of A cup G(B) when W empty}, we have the following.
\begin{corollary}\label{corollary n+1-W}
If $W\neq \emptyset$, then $|\mathcal A\cup G(\mathcal B)|\geq n +1 - |W|$. Moreover, we can insist that the sets in $G(\mathcal B)$ contain $W$ as a subset.
\end{corollary}
\begin{proof}
As before, it is enough to generate $\sum_{i=1}^k(|A_i|-1)+1$ elements in $G(\mathcal B)$ that contain $W$ as a subset. We split the analysis into two cases depending on whether all $A_i$ have size 1 or not. If there exists $l\in[k]$ such that $|A_l|>1$, this is identical to the proof of Case 2 in Corollary~\ref{Corollary bounding size of A cup G(B) when W empty}.

Suppose now that $|A_i|=1$ for all $i\in[k]$. It is enough to show that $\mathcal B$ is not empty, and moreover, it contains $X$ such that $W\subset X$. Clearly $W\notin\mathcal F$, and so it forms a diamond with 3 other elements of $\mathcal F$. Since $\emptyset\notin \mathcal F$, $W$ must be the minimal element of the diamond. Let $X\in\mathcal P([n])$ be maximal with this property, i.e. being the minimal element of a diamond in which the other 3 elements are in $\mathcal F$, and containing $W$ as a subset. Since $X\not\subseteq A$ for all $A\in\mathcal A$ as $W\subseteq X$, we have by construction that $X\in\mathcal B$, which completes the proof.
\end{proof}
\section{Proof of the main result}
The first part of this section is dedicated to showing that, if $|\mathcal F|\leq n$, then $\mathcal A\cup G(\mathcal B)$ and $\mathcal X\cup H(\mathcal Y)$ are disjoint, which, together with the bounds obtained in the previous section, will be directly used to obtain a contradiction.
\begin{proposition}\label{A and X}
Suppose that $|\mathcal F|<\frac{3n}{2}$. Then $\mathcal A$ and $\mathcal X$ are disjoint.
\end{proposition}
\begin{proof}
Suppose that there exists $X\in \mathcal A\cap \mathcal Y$. This means that $X$ is both a minimal and a maximal element of $\mathcal F$, which means that it is incomparable to all $F\in \mathcal F\setminus\{X\}$. Without loss of generality, by considering the diamond-saturated family $\mathcal F^c=\{[n]\setminus F : F \in \mathcal F\}$, we may assume that $|X|\geq \frac{n}{2}$.

By the minimality of $X$, $X\setminus\{x\}\not\in\mathcal F$ for all $x\in X$. Thus, there exist sets $P,Q,R\in \mathcal F$ such that $P$, $Q$, $R$ and $X\setminus\{x\}$ form a diamond. It is clear that $X\setminus\{x\}$ must be the minimal element of the diamond, as otherwise the minimal element will be strictly contained in $X$. Let $R$ be the maximal element of the diamond. We have that $X\neq P$ and $X\neq Q$, as otherwise we would have $X\subsetneq R$. Furthermore, $X\neq R$ as $X\setminus \{x\}\subsetneq P\subsetneq R$, hence $|R|\geq |X|+1$. This means that $P$, $Q$ and $R$ are all incomparable to $X$. Thus, since all contain $X\setminus \{x\}$, we get that $X\setminus P= \{x\}$, $X\setminus Q=  \{x\} $ and $X\setminus R  = \{x\}$. The sets analysed, and the Hasse diagram they form, are depicted below.
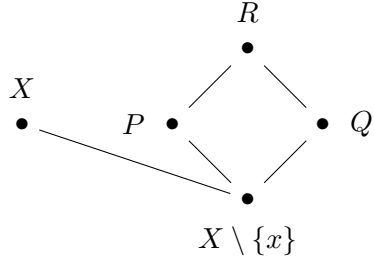
\begin{figure}[hbt!]
\hspace{0.1cm}\\
\centering
\begin{tikzpicture}
\node[label=above:$R$] (top) at (4,1) {$\bullet$}; 
\node[label=left:$P$] (left) at (3,0) {$\bullet$};
\node[label=right:$Q$] (right) at (5,0) {$\bullet$};
\node[label=above:\(X\)] (bokuno) at (1,0) {$\bullet$};
\node[label=below:$X\setminus\{x\}$] (bottom) at (4,-1) {$\bullet$};
\draw (top) -- (left) -- (bottom) -- (right) -- (top);
\draw (bottom) -- (bokuno);
\end{tikzpicture}
\end{figure}

Thus, for each $x\in X$, we have constructed three distinct elements $P_x,Q_x,R_x\in \mathcal F$ such that $X\setminus A = \{x\}$ for all $A\in \{P_x,Q_x,R_x\}$. This means that $|\mathcal F|\geq |\bigcup_{x\in X} \{P_x,Q_x,R_x\} | = 3|X|\geq\frac{3n}{2}$, a contradiction. Therefore, $\mathcal A\cap \mathcal X = \emptyset$, as claimed.
\end{proof}

\begin{proposition}\label{second disjoint lemma}
We have that $\mathcal A\cap H(\mathcal Y) = \emptyset$, and $\mathcal X\cap G(\mathcal B) = \emptyset$.
\end{proposition}
\begin{proof}
Suppose that there exists $X\in\mathcal A\cap H(\mathcal Y)$. By the definition of $H(\mathcal Y)$, there exist $Y\in\mathcal Y$ and $Z,Z'\in\mathcal F$ such that $X,Y,Z,Z'$ form a diamond, as depicted below.
\begin{figure}[hbt!]
\hspace{0.1cm}\\
\centering
\begin{tikzpicture}
\node[label=above:$Y$] (top) at (4,1) {$\bullet$}; 
\node[label=left:$X$] (left) at (3,0) {$\bullet$};
\node[label=right:$Z$] (right) at (5,0) {$\bullet$};
\node[label=below:$Z'$] (bottom) at (4,-1) {$\bullet$};
\draw (top) -- (left) -- (bottom) -- (right) -- (top);
\end{tikzpicture}
\end{figure}
Therefore, $Z'\subsetneq X$, and since $Z'\in\mathcal F$, $X$ is not a minimal element of $\mathcal F$, so $X\notin\mathcal A$, a contradiction. Thus $\mathcal A\cap H(\mathcal Y)=\emptyset$, and by symmetry, $\mathcal X\cap G(\mathcal B)=\emptyset$.
\end{proof}
\begin{proposition}\label{G H disjoint} Suppose that $|\mathcal F|\leq n$. Then $G(\mathcal A)\cap H(\mathcal Y)=\emptyset$.
\end{proposition}
\begin{proof}
Suppose that there exists $P\in G(\mathcal A)\cap H(\mathcal Y)$. Then, by definition, there exist elements $R,Q,X,Y\in\mathcal F$ such that we have the following diagram.
\begin{figure}[hbt!]
\hspace{0.1cm}\\
\centering
\begin{tikzpicture}
\node[label=above:$X$] (top) at (4,1) {$\bullet$}; 
\node[label=right:$P$] (left) at (3,0) {$\bullet$};
\node[label=right:$Y$] (right) at (5,0) {$\bullet$};
\node[label=below:$P\cap Y$] (bottom) at (4,-1) {$\bullet$};
\node[label=above:$P\cup R$] (top1) at (2,1) {$\bullet$}; 
\node[label=left:$R$] (right1) at (1,0) {$\bullet$};
\node[label=below:$Q$] (bottom1) at (2,-1) {$\bullet$};
\draw (top) -- (left) -- (bottom) -- (right) -- (top);
\draw (top1) -- (right1) -- (bottom1) -- (left) -- (top1);
\end{tikzpicture}
\end{figure}
Notice that the elements in $\mathcal B$ and $\mathcal Y$ have to be $P\cap Y$ and $P\cup R$ respectively, by the maximality and minimality conditions of $\mathcal B$ and $\mathcal Y$ respectively.

Since $P\cap Y\in\mathcal B$, by Lemma~\ref{sizeB} we get that, for all $i\notin P\cap Y$, there exists $X_i\in\mathcal F$ such that $X_i\subseteq (P\cap Y)\cup\{i\}$ and $i\in X_i$. We choose exactly one such set for each $i\notin P\cap Y$, and let $S_1=\{X_i:i\notin P\cap Y\}$.

Similarly, since $P\cup R\in\mathcal Y$, by Lemma~\ref{sizeB} we get that, for all $i\in P\cup R$, there exists $Y_i\in\mathcal F$ such that $(P\cup R)\setminus\{i\}\subseteq Y_i$ and $i\notin Y_i$. We choose exactly one such set for each $i\in P\cup R$, and let $S_2=\{Y_i:i\notin P\cap Y\}$.

Then $n\geq|\mathcal F|\geq|S_1\cup S_2|=|S_1|+|S_2|-|S_1\cap S_2|=n-|P\cap Y|+|R\cup P|-|S_1\cap S_2|$. Coupled with the fact that $|P\cap Y|\leq |P|-1$ and $|P\cup R|\geq |P| + 1$, we get that $|S_1 \cap S_2|\geq 2$.

Let $Z\in S_1\cap S_2$. Since the sets in $S_1$ have size at most $|P\cap Y|+1$, and the sets in $S_2$ have size at least $|P\cup R|-1$, we get that $|Z|=|P|$, and that there exists $a\notin P$ and $b\in P$ such that $P\cup R=P\cup\{a\}$ and $P\cap Y=P\setminus\{b\}$.

Since $Z\in S_1$, let $i\notin P\setminus\{b\}$ be such that $Z\subseteq (P\setminus\{b\})\cup\{i\}$ and $i\in Z$. Also, since $Z\in S_2$, let $j\in P\cup\{a\}$ be such that $(P\cup\{a\})\setminus\{j\}\subseteq Z$ and $j\notin Z$. We therefore have $$(P\cup\{a\})\setminus\{j\}\subseteq Z\subseteq (P\setminus\{b\})\cup\{i\}.$$
Therefore, either $i=a$, $j=b$, and $Z=(P\setminus\{b\})\cup\{a\}$, or $i=b$, $j=a$ and $Z=P$. Putting everything together, $S_1\cap S_2=\{P, (P\setminus\{b\})\cup\{a\}\}$, and $|S_1\cup S_2|=n$.

However, $X\in\mathcal F$, and since $P\subsetneq X$, $|P|<|X|$, thus $X\notin S_1$. Furthermore, if $X\in S_2$ and it corresponds to $l\in P\cup\{a\}$, then $(P\cup\{a\})\setminus\{l\}\subseteq X$ and $l\notin X$. But since $P\subset X$, $l\notin P$, and so $l=a$. This is a contradiction since $P\in S_2$, corresponds to $a$, and $P\neq X$. 

This tells us that $X\notin S_1\cup S_2$, and so $|\mathcal F|\geq|S_1\cup S_2\cup\{X\}|=n+1$, a contradiction. Thus, $G(\mathcal A)\cap H(\mathcal Y)=\emptyset$.
\end{proof}

By \Cref{A and X,second disjoint lemma,G H disjoint}, we indeed have that, if $|\mathcal F|\leq n$, then $\mathcal A\cup G(\mathcal B)$ and $\mathcal X\cup H(\mathcal Y)$ are disjoint. We are now ready to prove our main result. The argument will be split into two cases, depending on whether $W\cup\overline{W}=\emptyset$ or not.
\begin{proof}[Proof of Theorem~\ref{maintheorem}] Since a full chain has size $n+1$ and it is diamond-saturated, we have that $\text{sat}^{*}(n,\mathcal D_2)\leq n+1$. Now let $\mathcal F$ be a diamond-saturated family. As discussed previously, if one of $\emptyset$ or $[n]$ are in $\mathcal F$, we have that $|\mathcal F|\geq n+1$. Thus we may assume that $\emptyset,[n]\notin\mathcal F$, and so all our work, notations and setup from previous sections applies. Suppose that $|\mathcal F|\leq n$. By the above $\mathcal A\cup G(\mathcal B)$ and $\mathcal X\cup H(\mathcal Y)$ are disjoint. We have two cases.
\begin{case2} $W=\overline{W}=\emptyset$.
\end{case2}
By Corollary~\ref{Corollary bounding size of A cup G(B) when W empty}, we have that $$|\mathcal F|\geq |\mathcal A\cup G(\mathcal B)|+|\mathcal X\cup H(\mathcal Y)|\geq n+1-m_{\mathcal A}+n+1-\max_{X\in\mathcal X}(n-|X|).$$
In other words, $|\mathcal F|\geq n+2+\min_{X\in\mathcal X}|X|-\max_{A\in\mathcal A}|A|$. However, $\min_{X\in\mathcal X}|X|\geq\max_{A\in\mathcal A}|A|$. Indeed, suppose that that is not the case. Let $A\in\mathcal A$ be of maximal cardinality, and let $X\in\mathcal X$ be of minimal cardinality. By our assumption we have that $|A|\geq |X|+1$. By Lemma~\ref{sizeA} we get that there exist at least $|A|$ elements $F\in\mathcal F$ such that $|F|\geq|A|$, and that there exist at least $n-|X|$ elements $G\in\mathcal F$ such that $|G|\leq |X|$. Since, by size, these two collections of sets do not intersect, we get that there are at least $|A|+n-|X|\geq n+1$ elements in $\mathcal F$, a contradiction.

Thus $\min_{X\in\mathcal X}|X|\geq\max_{A\in\mathcal A}|A|$, and consequently $|\mathcal F|\geq n+2$, a contradiction.
\begin{case2} $W\neq\emptyset$, or $\overline{W}\neq\emptyset$.
\end{case2}
Suppose without loss of generality that $W\neq \emptyset$, and let $\mathcal F_1=\{B\in G(\mathcal B):W\subseteq B\}$. By Corollary~\ref{corollary n+1-W} we have that $|\mathcal F_1|+|\mathcal A|=|\mathcal F_1\cup\mathcal A|\geq n+1-|W|$. By Lemma~\ref{AandBincomparable}, there exists $A\in\mathcal A$ and $X\in\mathcal X$ such that $A\not\subseteq B$.

Now, by Lemma~\ref{nicelemmame}, for every $i\in W$, there exists $X_i, X_i\cup\{i\}\in\mathcal F$ such that $A\subseteq X_i$. Let $\mathcal F_2=\{X_i, X_i\cup\{i\}:i\in W\}\subseteq\mathcal F$. Consider the graph with vertex set the elements of $\mathcal F_2$, and an edge between each $X_i$ and $X_i\cup\{i\}$. It is easy to check that this graph is acyclic, hence a forest. Let $t$ be the number of connected components. Then the number of edges is at most the number of vertices minus $t$. In other words, $|W|\leq |\mathcal F_2|-t$.

If within a connected component there exists a vertex that contains $W$ as a subset, then it is the only one with this property as we cannot reach another set that still contains all of $W$ by repeatedly adding or removing one element of $W$ at a time. Therefore, every connected component meets $\mathcal F_1$ at most once, which tells us that $|\mathcal F_1\cup\mathcal F_2|\geq|\mathcal F_1|+|W|$. Moreover, since every element of $\mathcal F_2$ contains $A\in\mathcal A$, $|\mathcal F_2\cap\mathcal A|\leq 1$. Together with the above and the fact that $\mathcal A\cap\mathcal F_1=\emptyset$, we have that $|\mathcal F|\geq|\mathcal F_1\cup\mathcal F_2\cup\mathcal A|\geq |\mathcal F_1|+|W|+|\mathcal A|-1\geq n$.

Finally, looking at $X\in\mathcal X\subseteq\mathcal F$, we clearly have by Proposition~\ref{second disjoint lemma} that $X\notin\mathcal F_1\subseteq G(\mathcal B)$, and by Proposition~\ref{A and X} $X\notin\mathcal A$. Furthermore, since every element of $\mathcal F_2$ contains $A$, and $A$ and $X$ are incomparable, $X\notin\mathcal F_2$. Thus, $X\notin\mathcal F_1\cup\mathcal F_2\cup\mathcal A$, which yields at least $n+1$ elements in $\mathcal F$, a contradiction.

Therefore, in all cases, $|\mathcal F|\geq n+1$, which gives that $\text{sat}^*(n,\mathcal D_2)=n+1$, as claimed.
\end{proof}
\section{Concluding remarks}
The above proof shows that the size of a diamond-saturated family cannot go below $n+1$, but it does not tell us much about the structure of such set-systems. In particular, aside from a maximal chain, the empty set with all the singletons, or the full set with the complements of singletons, there are no other known diamond-saturated families of size exactly $n+1$. Is it possible that these are the only ones? Does there exist a diamond-saturated family of size $n+1$ that does not contain $\emptyset$ or $[n]$? We conjecture the following.
\begin{conjecture}
Let $\mathcal F$ be a diamond-saturated family. If $|\mathcal F|=n+1$, then $\mathcal F$ is either a maximal chain, the empty set together with all singletons, or the full set together with all complements of singletons. Moreover, if $\emptyset,[n]\notin\mathcal F$, then $|\mathcal F|\geq 2n-c$, for some universal constant $c$.
\end{conjecture}

Since the diamond is the 2-dimensional hypercube, it is natural to wonder what is the saturation number for a general hypercube $Q_k$, viewed as a poset. For example, for $Q_3$ we know that $\text{sat}^*(n,Q_3)\leq 3n-2$, achieved by $\mathcal F=\{\{i\},\{1,j\}, \{2,j\}:i\in[n],3\leq j\leq n\}\cup\{\emptyset,[n]\}$. Is this tight? What about a general $Q_k?$ We propose the following.
\begin{conjecture}
Let $k\geq2$. Then $\text{sat}^*(n, Q_k)=(2^{k-1}-1)n-c$, for some absolute constant $c$.
\end{conjecture}
\bibliographystyle{amsplain}
\bibliography{references.bib}
\Addresses
\end{document}